\documentclass[12pt]{amsart}
\usepackage{amsmath}
\usepackage{amssymb}
\usepackage{ifthen}
\usepackage[english]{babel}
\usepackage{epsfig}
\nonstopmode

\usepackage{mathrsfs}
\usepackage{graphicx}
\usepackage{latexsym}
\usepackage{verbatim}

\numberwithin{equation}{section}

\newcommand {\mo} {\,{\rm mod}\,}
\newcommand{\diam}{{\rm diam}\,}
\newcommand{\dist}{{\rm dist}\,}
\newcommand{\Int}{{\rm Int}\,}

\renewcommand{\Im}{{\rm Im\,}}

\newcommand {\Sn} {{\overline{\mathbb R}^n}}
\newcommand {\Hn} {{\mathbb H^n}}
\newcommand {\ring} {{\mathcal R}}
\newcommand {\T} {{\mathcal T}}
\newcommand {\es} {{\mathcal S}}
\newcommand {\C} {\mathbb C}
\newcommand {\R} {\mathbb R}
\newcommand {\B} {\mathbb B}
\newcommand {\K} {\mathbf K}

\newcommand{\M}{\mathsf{M}}

\newcommand {\A} {\mathcal A}
\newcommand {\aand} {\quad\text{and}\quad}
\newcommand {\ep} {\varepsilon}
\def\QED{{\hfill $\Box$\par\medskip}}

\theoremstyle{theorem}

\numberwithin{equation}{section}

\newenvironment{thm}[1][]{%
\refstepcounter{equation}%
\medskip%
\noindent%
\arabic{section}.\arabic{equation} {\bf Theorem}%
\ifthenelse{\equal{#1}{}}{}{ (#1)}%
{\bf .}
\itshape}{\medskip}

\newenvironment{cor}[1][]{%
\refstepcounter{equation}%
\medskip%
\noindent%
\arabic{section}.\arabic{equation} {\bf Corollary}%
\ifthenelse{\equal{#1}{}}{}{ (#1)}%
{\bf .}
\itshape}{\medskip}

\newenvironment{lem}[1][]{%
\refstepcounter{equation}%
\medskip%
\noindent%
\arabic{section}.\arabic{equation} {\bf Lemma}%
\ifthenelse{\equal{#1}{}}{}{ (#1)}%
{\bf .}
\itshape}{\medskip}

\newenvironment{prop}[1][]{%
\refstepcounter{equation}%
\medskip%
\noindent%
\arabic{section}.\arabic{equation} {\bf Proposition}%
\ifthenelse{\equal{#1}{}}{}{ (#1)}%
{\bf .}
\itshape}{\medskip}

\renewcommand{\subsection}[1]%
{\smallskip \noindent \refstepcounter{equation}\arabic{section}.\arabic{equation} {\bf #1.}}

\begin{document}
\title[Teichm\"uller's theorem in higher dimensions]{Teichm\"uller's theorem in higher dimensions \\ and its applications}

\author[A. Golberg]{Anatoly Golberg}
\address{Department of Mathematics \\
Holon Institute of Technology \\
52 Golomb St., P.O.B. 305, Holon 5810201, Israel}
\email{golberga@hit.ac.il}

\author[T. Sugawa]{Toshiyuki Sugawa}
\address{Graduate School of Information Sciences \\
Tohoku University \\
Aoba-ku, Sendai 980-8579, Japan}
\email{sugawa@math.is.tohoku.ac.jp}

\author[M. Vuorinen]{Matti Vuorinen}
\address{Department of Mathematics and Statistics, University of Turku, FI-20014 Turku, Finland}
\email{vuorinen@utu.fi}

\begin{abstract}
For a given ring (domain) in $\overline{\mathbb{R}}^n$ 
we discuss whether its boundary components can be separated by an annular ring
with modulus nearly equal to that of the given ring.
In particular, we show that, for all $n\ge 3\,,$ the standard definition of
uniformly perfect sets in terms of Euclidean metric is equivalent
to the boundedness of moduli of separating rings.
We also establish separation theorems for a ``half" of a ring.
As applications of those results, we will prove boundary H\"older continuity
of quasiconformal mappings of the ball or the half space in $\mathbb{R}^n.$
\end{abstract}
\maketitle


{\small{\em{\centerline{Dedicated to the memory of Professor Stephan Ruscheweyh}}}}

\medskip

\bigskip

\section{Introduction}
A doubly connected domain $\ring$ in the complex plane $\C$ is called a ring domain
or, simply, a ring.
By the Uniformization Theorem, the ring $\ring$ is conformally equivalent to the annulus
$\{z\in\C: r_0<|z|<r_1\}$ for some $0\le r_0<r_1\le\infty.$
We will exclude here the doubly degenerate case when $r_0=0$ and $r_1=\infty.$
The quantity $\log(r_1/r_0)$ is called the modulus of $\ring$ and denoted by
$\mo\ring.$
(Note that in the literature  the modulus is sometimes defined as $ \frac{1}{2 \pi}\log (r_1/r_0)\,.$)
O.~Teichm\"uller \cite{Teich} showed that a ring $\ring$ with $\mo\ring>\pi$
separating $0$ and $\infty$ contains a circle centered at $0$ and that
the constant $\pi$ is sharp (see also \cite{ABP}).
Indeed, the Teichm\"uller ring $R_T(t)=\C\setminus\big([-1,0]\cup [t,+\infty)\big)$
with $t=1$ serves as an extremal case.
Teichm\"uller introduced the Gr\"otzsch ring and the Teichm\"uller ring and
found their extremal properties in \cite{Teich}
(see Lemmas \ref{lem:GR} and \ref{lem:TR} below for details).
Using the extremal property of the Teichm\"uller ring,
D.~A.~Herron, X.~Liu and D.~Minda \cite{HLM89} showed the following sharp result.

\begin{thm}[Herron-Liu-Minda]
Let $\ring$ be a ring separating $0$ and $\infty$ in $\C$ with $m=\mo\ring>\pi.$
Then $\ring$ contains an annular subring $\A$ of the form $\{z: r_0<|z|<r_1\}$
with
$$
\mo\A=\log \mu_T^{-1}(m),
$$
where $\mu_T(t)=\mo R_T(t)$ for $0<t<+\infty.$
The result is sharp.
\end{thm}

From the inequality $\mu_T(t)<\log t+\pi$ for $t>1$ (see Lemma \ref{lem:mono} below), which is equivalent to $m<\log\mu_T^{-1}(m)+\pi$ for $m=\mu_T(t)>\pi$,
F.~G.~Avkhadiev and K.-J.~Wirths deduced a sharp explicit form of
the above theorem (see Theorem \ref{thm:AW} below).
For convenience, in this paper, we use the term Teichm\"uller's theorem
for this sort of separation results.
A subset $\es$ of $\C$ is called a semiring if it is homeomorphic
to the upper half $\{z\in\C: r_0\le |z|\le r_1, \Im z>0\}$ of the
closed annulus $r_0\le |z|\le r_1$ for some $0<r_0<r_1<\infty.$
V.~Gutlyanski\u\i, K.~Sakan and T.~Sugawa established a result
similar to Teichm\"uller's theorem for semirings in $\C$
and applied it to the study of boundary regularity of homeomorphisms
of the unit disk or the upper half-plane.

Our main goal in the present paper is to extend those results
to higher dimensions.
Indeed, the higher dimensional analogs of the Gr\"otzsch ring and
the Teichm\"uller ring were intensively studied (see, for instance,
\cite{AVV97} and \cite{GMP17}) and found many important applications
in the theory of quasiconformal and quasiregular mappings in higher dimensions.
However, it seems that higher dimensional analogs of Teichm\"uller's theorem are less known. 
One of such extremal problems is to find annular rings of the largest modulus 
which separate two pairs of points in $\overline{\mathbb R}^n$ and this problem has been studied in 
\cite{Bet03} and \cite{Ibr03} independently.
We will extend Teichm\"uller's theorem and its semiring counterpart
to higher dimensions and, as examples of applications, we give a
conformally invariant characterization of uniformly perfect sets
in $\Sn$ and we will give (at least conceptually) simple proofs for the known fact
that quasiconformal self-homeomorphisms of open balls or half-spaces
extend to the boundary in a H\"older continuous way.
We emphasize that our approach may allow us to weaken regularity or quasiconformality assumptions
of the mappings.
Such applications to mappings of finite directional dilatations will be presented in
our forthcoming paper.


\section{Gr\"otzsch and Teichm\"uller rings and related estimates}

\subsection{Modulus of curve family}
We denote by $\Sn$ the extended Euclidean $n$-space $\R^n\cup\{\infty\},$
which is homeomorphic to the $n$-sphere $\mathbb{S}^n.$
Throughout the paper, we will assume that $n$ is an integer greater than $1.$
Let $\Gamma$ be a family of curves in $\Sn.$
A Borel measurable function $\rho$ on $\R^n$ is called admissible for $\Gamma$
if $0\le \rho\le+\infty$ almost everywhere on $\R^n$ and if
$\int_\gamma\rho(x)|dx|\ge1$ for every locally rectifiable $\gamma\in\Gamma.$
The (conformal) modulus of $\Gamma$ is defined to be
$$
\M(\Gamma)=\inf_\rho \int_{\R^n}\rho(x)^n dm(x),
$$
where the infimum is taken over all admissible functions $\rho$ on $\R^n$ for $\Gamma$
and $dm$ is the Lebesgue measure on $\R^n.$

\subsection{Rings}
In the paper, a continuum will mean a connected, compact and non-empty set.
We call it non-degenerate if it contains more than one point.
A continuum $C\subsetneq\Sn$ is called {\it filled} if $\Sn\setminus C$ is connected.
For a pair of disjoint filled continua $C_0$ and $C_1$ in $\Sn,$
the set $\ring=\Sn\setminus(C_0\cup C_1)$ is open and connected and
will be called a {\it ring} and sometimes denoted by $\ring(C_0, C_1).$
The ring $\ring$ is said to have nondegenerate boundary if each component $C_j$
contains at least two points.
We will say that $\ring(C_0,C_1)$ separates a set $E$ if $\ring\cap E=\varnothing$ and if
$C_j\cap E\ne\varnothing$ for $j=0,1.$
In the sequel, when $\ring\subset\R^n,$
we will assume conventionally that $\infty\in C_1$ unless otherwise stated.

Let $\Gamma_\mathcal R$ be the family of all curves joining $C_0$ and $C_1$ in $\mathcal R.$
Then the modulus (called also the module) of  $\mathcal R$ is defined by
\begin{equation*}
\mo \mathcal R=\left[\frac{\omega_{n-1}}{\M(\Gamma_\mathcal R)}\right]^{1/(n-1)},
\end{equation*}
where $\omega_{n-1}$ denotes the area of the unit $(n-1)$-dimensional sphere.
More precisely,
$$
\omega_{n-1}=\frac{n\pi^{n/2}}{\Gamma((n/2)+1)}
=\begin{cases}
\dfrac{2k\pi^k}{k!} &~\text{if}~ n=2k, \\
\null & \vspace{-3mm}\\
\dfrac{(2k-1)k!2^{2k}\pi^{k-1}}{(2k)!} & ~\text{if}~ n=2k-1.
\end{cases}
$$
For the annular ring $\A(a; r_0,r_1)=\{x\in\mathbb R^n\,:\,r_0<|x-a|<r_1\},$
we ha\-ve $\mo \A(a;r_0,r_1)=\log(r_1/r_0)$ (see, for instance, \cite[pp. 22-23]{Vai71}).

A ring $\ring'$ is said to be a {\it subring} of a ring $\ring$ if $\ring'\subset\ring$ and if
each component of $\Sn\setminus\ring'$ intersects $\Sn\setminus\ring.$
By the monotonicity of the moduli of curves, we have the inequality $\mo \ring'\le \mo\ring.$

\subsection{Gr\"otzsch and Teichm\"uller rings}
Two canonical rings are of special interest because of 
the extremal properties of their moduli.
The first one is the Gr\"otzsch ring $R_{G,n}(s),$ $s>1,$ and defined by
\begin{equation*}
R_{G,n}(s)=\mathcal R(\overline{\mathbb B}^n, [se_1,\infty]).
\end{equation*}
Here $\mathbb B^n$ is the unit ball centered at the origin,
$\overline{\mathbb B}^n$ is its closure, and $e_1$ is the unit vector $(1,0,\ldots,0)$
in $\R^n.$
The second one is the Teichm\"uller ring $R_{T,n}(t),$ $t>0,$ and defined by
\begin{equation*}
R_{T,n}(t)=\mathcal R([-e_1,0], [te_1,\infty]).
\end{equation*}
The functions $\gamma_n(s)=\M(\Gamma_{R_{G,n}(s)})$ and
$\tau_n(t)=\M(\Gamma_{R_{T,n}(t)})$ are intensively studied in \cite{AVV97}.
For example, these functions are strictly decreasing and continuous functions.
The Gr\"otzsch and Teichm\"uller rings have the following extremal properties.
See \cite[5.4.1, pp. 181-182]{GMP17} for their proofs.

\begin{lem}\label{lem:GR}
Let $\ring$ be the ring $\ring(\overline\B^n,C_1)$
for a filled continuum $C_1$ with $y, \infty\in C_1$ in the domain $|x|>1.$
Then the inequality
$$
\mo\ring\le\mo R_{G,n}(|y|)
$$
holds; equivalently,
$$
\M(\Gamma_\ring)\ge \gamma_n(|y|).
$$
\end{lem}

\begin{lem}\label{lem:TR}
For filled continua $C_0, C_1$ with $0, -e_1\in C_0$ and $x_1,\infty\in C_1,$ the following inequality holds:
\begin{equation*}
\mo \mathcal R(C_0,C_1)\le \mo R_{T,n}(|x_1|).
\end{equation*}
\end{lem}

The relation between the moduli of these special rings can be written
\begin{equation*}
\mo R_{T,n}(t)\,=\,2 \mo R_{G,n}(s),\qquad s=\sqrt{t+1}\,.
\end{equation*}
See \cite{AVV97}, \cite{Zor17}.
The real-valued functions $\Phi_n$ and $\Psi_n$ defined by
\begin{align*}
\log\Phi_n(s)&=\mo R_{G,n}(s)=\left[\frac{\omega_{n-1}}{\gamma_n(s)}\right]^{1/(n-1)}, \\
\log\Psi_n(t)&=\mo R_{T,n}(t)=\left[\frac{\omega_{n-1}}{\tau_n(t)}\right]^{1/(n-1)}\,,
\end{align*}
are of a special interest and have frequent applications in Complex and Real Analysis.
When $n=2,$ explicit forms of $\mu_G(t)=\log\Phi_2(s)$ and $\mu_T(t)=\log\Psi_2(t)$ are known.
Indeed,
\begin{equation}\label{eq:GT}
\mu_G(s)=\mu\left(\frac1s\right)
\aand
\mu_T(t)=2\mu\left(\frac1{\sqrt{t+1}}\right)=\pi\cdot\frac{\K(\frac t{t+1})}{\K(\frac1{t+1})},
\end{equation}
where $\mu(r)=(\pi/2)\K(1-r^2)/\K(r^2)$ and
$\K(w)$ denotes the complete elliptic integral of the first kind
(see \cite[Chap.~5]{AVV97} for details):
$$
\K(w)=\int_0^1\frac{dx}{\sqrt{(1-x^2)(1-wx^2)}}.
$$

\subsection{Basic properties of $\Phi_n(s)$ and $\Psi_n(t)$}
Following \cite{AVV97} and \cite{GMP17},
we recall several properties of the above quantities:

\noindent (a) the function $s\mapsto \Phi_n(s)/s$ is nondecreasing on $(1,\infty);$

\noindent (b) $\lim_{s\to \infty} \Phi_n(s)=\infty;$

\noindent (c) $\lim_{s\to 1^+} \Phi_n(s)=1;$

\noindent (d) the Gr\"otzsch (ring) constant $\lambda_n:= \lim_{s\to \infty} \Phi_n(s)/s$ exists in $(1,\infty];$

\noindent (e) $4\le\lambda_n \le 2^{n/(n-1)}e^{n(n-2)/(n-1)};$

\noindent (f) $\lambda_2=4$ and the exact value of $\lambda_n$ is unknown for $n\ge 3$;

\noindent (g) the function $\Phi_n$ is strictly increasing and continuous on the interval $(1,\infty),$ and
\begin{equation*}
s\,\le\,\Phi_n(s)\,\le\,\lambda_n s;
\end{equation*}

\noindent (h) the function $\Psi_n$ is strictly increasing and continuous on the interval $(0,\infty),$ and
\begin{equation}\label{estTr}
t+1\,\le\,\Psi_n(t)\,\le\,\lambda_n^2 (\sqrt{1+t}+\sqrt{t})^2/4;
\end{equation}


\noindent (i) if $R_E(n,a)=\mathcal R(C_0, C_1),$ where $C_1=\left\{x\in\mathbb R^n: \frac{x_1^2}{a^2+1}+\frac{x_2^2}{a^2}+\ldots+\frac{x_n^2}{a^2}\ge 1\right\}$ with $a>1$
and $C_0=\{te_1: t\in [-1,1]\},$ the modulus of this ring admits the following upper bound in terms of elliptic integrals
\begin{equation*}
\mo R_E(n,a)\,\le\,\int\limits_1^b\left(\frac{r^2+1}{r^2-1}\right)^{(n-2)/(n-1)}\frac{dr}{r},\quad b=a+\sqrt{a^2+1}\,,
\end{equation*}
for the planar case we have the equality
\begin{equation*}
\mo R_E(2,a)\,=\,\log b,
\end{equation*}
and the Gr\"otzsch constant can be found by
\begin{equation*}
\log \lambda_n\,=\,\lim\limits_{a\to\infty}\left(\mo R_E(n,a)-\log\frac{a}{2}\right)\,.
\end{equation*}

We take this opportunity to give another proof of the inequality
used by Avkhadiev and Wirths to show Theorem \ref{thm:AW} below.
Note that they used an infinite product expansion of the inverse
function $\mu_T^{-1}(m).$

\begin{lem}\label{lem:mono}
The function $\Psi_2(t)/t$ is strictly decreasing in $0<t<+\infty.$
Also the inequality $\mu_T(t)<\log t+\pi$ holds for $t>1.$
\end{lem}

\begin{proof}
We consider the function
$$
g(t)=\log\frac{\Psi_2(t)}t=\mu_T(t)-\log t
=2\mu\left(\frac1{\sqrt{t+1}}\right)-\log t.
$$
Put $r=1/\sqrt{t+1}\le 1/\sqrt 2$ for brevity.
Then $dr/dt=-r^3/2$ and $t=(1-r^2)/r^2.$
We now use the formula
$\mu'(r)=-\pi^2/\{4r(1-r^2)\K(r^2)^2\}$ \cite[(5.9)]{AVV97} to get
$$
g'(t)=-\frac{r^3}2\mu'(r)-\frac1t
=\frac{r^2}{1-r^2}\left(\frac{\pi^2}{4\K(r^2)^2}-1\right).
$$
Since ${\K}(r^2) >{\K}(0)= \pi/2 $ for all $r \in (0,1)\,,$
we have $g'(t) <0$ for $t>0\,.$
We thus conclude that $g(t)$ and $\Psi_2(t)/t$ are strictly decreasing in $0<t<+\infty.$
In particular, we have $\mu_T(t)-\log t=g(t)<g(1)=\pi$ for $t>1.$
This is equivalent to the second assertion.
\end{proof}

\section{Extension of Teichm\"uller's theorem}

\subsection{Teichm\"uller's theorem}
Roughly speaking, Teichm\"uller's theorem states that a ring in $\R^2=\C$ with modulus at least a certain number
should contain an annular ring of equal modulus up to a bounded term.
There are various versions of results of similar nature.
One of the most convenient results is the following theorem due to Avkhadiev and Wirths \cite[Theorem~3.17]{AW09}
(see also \cite{Sug10}).

\begin{thm}[Avkhadiev-Wirths]\label{thm:AW}
Let $\ring$ be a ring in $\C$ with $\mod \ring>\pi$ which separates
a given point $z_0\in\C$ from $\infty.$
Then there is a subring $\A$ of $\ring$ which has the form $\{z: r_0<|z-z_0|<r_1\}$
and satisfies the condition $\mod \A\ge\mod \ring-\pi.$
The constant $\pi$ cannot be replaced by any smaller number.
\end{thm}

\subsection{Extension to higher dimension}\label{ssec:An}
One of our main tools in the present paper is an analogue of Teichm\"uller's lemma in higher dimensions.
A ring $\A$ of the form $\A(a; r_0, r_1)=\{x\in\R^n: r_0<|x-a|<r_1\}$ for some
$0<r_0<r_1<+\infty$ and $a\in\R^n$ will be called an {\it annular ring} or an {\it annulus} (centered at $a$).
We recall that $\mo\A(a; r_0,r_1)=\log(r_1/r_0).$
Let
$$
A_n
=\sup_{1<t<+\infty}\big[\mo R_{T,n}(t)-\log t\big]
=\sup_{1<t<+\infty}\log\frac{\Psi_n(t)}{t}.
$$
Then we obtain the following theorem.

\begin{thm}\label{thm:teich}
Let $\mathcal R$ be a ring in $\overline{\mathbb R}^n$ separating a given point $x_0\in\R^n$
and $\infty$ and satisfying the condition $\mo \mathcal R>A_n.$
Then there exists an annular subring $\A$ of $\ring$ centered at $x_0$ with
$\mo \A\ge \mo \ring - A_n.$
The constant $A_n$ is sharp.
Moreover, the number $A_n$ admits the estimate
\begin{equation}\label{eq:An}
A_n\le 
2\log\frac{(1+\sqrt 2)\lambda_n}2
=\log \frac{(3+2\sqrt{2})\lambda_n^2}{4}.
\end{equation}
\end{thm}

When $n=2,$ by Lemma \ref{lem:mono}
$\Psi_2(t)/t$ is decreasing in $t>1$ so that $A_2=\log\Psi_2(1)=\pi,$
cf.~\eqref{eq:GT}.
Thus the theorem reduces to Theorem \ref{thm:AW} if $n=2.$
On the other hand, \eqref{eq:An} gives $A_2\le 2\log 2(1+\sqrt 2)\approx 3.14904.$
This bound already appeared in Corollary 3.5 of the paper \cite{HLM89} by Herron, Liu and Minda.
If we could show that $\Psi_n(t)/t$ is non-increasing in $1\le t<+\infty,$ 
we will have $A_n=\log\Psi_n(1).$

\begin{proof}
Let $\ring=\ring(C_0,C_1)$ with $\infty\in C_1.$
Put $r_0=\max\limits_{x\in C_0} |x-x_0|.$
By performing a suitable affine transformation,
one may assume without loss of generality that $x_0=0,$ $r_0=1,$ $-e_1\in C_0.$
Let $r_1=\exp(\mo\mathcal R -A_n)>1.$
Now we show that the annular ring $\A=\{1<|x|<r_1\}$ separates $C_0$ from $C_1.$
Clearly, $\A\cap C_0=\varnothing.$
Suppose, on the contrary, that $\A\cap C_1\ne\varnothing.$
That is, there is a point $x_1\in C_1$ with $|x_1|<r_1.$
By Lemma \ref{lem:TR} and the strict monotonicity of $\Psi(t)$ we have
\begin{equation*}
\log r_1+A_n=\mo \mathcal R\,\le\,\mo R_{T,n}(|x_1|)<\log\Psi(r_1).
\end{equation*}
This implies $A_n<\log\Psi(r_1)-\log r_1$ which does not agree with the definition of $A_n.$
Hence, we have shown that $\A$ is a subring of $\ring$ as required.

We next show that the constant $A_n$ cannot be replaced by a smaller one.
Let $0<A<A_n.$
Then there is a $t_0\in(1,\infty)$ such that $A<\log\Psi(t_0)-\log t_0<A_n.$
We now take $R_{T,n}(t_0)$ as $\ring.$
The maximal annular subring of $\ring$ centered at $0$ is obviously $\A=\{x: 1<|x|<t_0\}$
and the inequality $\mo \A=\log t_0<\log\Psi(t_0)-A=\mo\ring-A$ holds.
Therefore, we cannot replace $A_n$ by $A$ in the assertion of the theorem.

Finally we show \eqref{eq:An}.
By \eqref{estTr}, we observe that
$$
\frac{\Psi_n(t)}{t}\le \left(\frac{(\sqrt{1+t}+\sqrt t)\lambda_n}{2\sqrt{t}}\right)^2
<\left(\frac{(\sqrt2+1)\lambda_n}2\right)^2
$$
for $t>1.$
Hence,
$$
A_n\le 2\log\frac{(1+\sqrt 2)\lambda_n}2.
$$
\end{proof}

\subsection{Uniform perfectness}
A closed subset $E$ of $\Sn$ containing at least two points is said to be {\it uniformly perfect}
if there exists a constant $0<c<1$ such that
\begin{equation}\label{eq:UP}
\{x\in E: cr< |x-a|< r\}\ne\varnothing \quad
\text{for}~ a\in E\setminus\{\infty\},~0<r<\diam E.
\end{equation}
Here, we denote by $\diam E$ the Euclidean diameter of $E$ and set $\diam E=\infty$
when $\infty\in E.$
We can characterize uniformly perfect sets in a conformally invariant manner.
The planar case is classical, see \cite{SugawaUP} or \cite{AW09}.
For more information about uniformly perfect sets in $\Sn$ the reader may look at
\cite{JV96} (also \cite{SugawaUPS} for a survey and \cite[pp. 343-345]{gm} for many alternative
characterizations).

\begin{thm}\label{thm:sep0}
A closed set $E$ in $\Sn$ with $\sharp E\ge 2$ is uniformly perfect
if and only if there exists a constant $M>0$ such that an arbitrary ring $\ring$ in $\Sn$
which separates $E$ satisfies the inequality $\mo\ring\le M.$
\end{thm}

As we will see in the proof later, condition \eqref{eq:UP} implies $M\le A_n+\log(3/c),$
where $A_n$ is given in \S \ref{ssec:An}.
As a preparation of the proof, we first show the following lemma.

\begin{lem}\label{lem:ann}
Let $\A= \{x: r_0<|x-a|<r_1\}$ be an annular ring in $\R^n$ separating the origin from $\infty$ with $\mo\A>\log3.$
Then $\A'=I(\A)$ contains an annular subring $\A_0$ with $\mo\A_0\ge\mo\A-\log3,$
where $I$ is the reflection in the unit sphere: $I(x)=x/|x|^2.$
Moreover, $a'=I(a)$ $($the origin $0)$ can be chosen to be the center of $\A_0$
as the unbounded component of $\Sn\setminus\A$ $($the bounded component of $\Sn\setminus\A$,
respectively$)$ contains the origin.
\end{lem}

\begin{proof}
We first consider the case when $0\in C_1;$ equivalently, $|a|\ge r_1.$
Suppose that $x\in\R^n$ is on the sphere $|x-a|=r$ with $|a|>r.$
Let $u=x-a,~ x'=I(x)$ and $a'=I(a).$
Then $|u|=r$ and
\begin{align*}
|x'-a'|^2&=|x'|^2-2\,x'\cdot a'+|a'|^2 \\
&=\frac1{|x|^2}-2\frac{x\cdot a}{|x|^2|a|^2}+\frac1{|a|^2} \\
&=-\frac1{|x|^2}-2\frac{u\cdot a}{|x|^2|a|^2}+\frac1{|a|^2}.
\end{align*}
Letting $u\cdot a=t|a|^2$ with $|t|\le r/|a|,$ we obtain
$$
|x'-a'|^2 =\frac1{|a|^2}-\frac1{|x|^2}-\frac{2t}{|x|^2}
=\frac1{|a|^2}-\frac{1+2t}{|a|^2(1+2t)+r^2}=: h(t).
$$
Since $h(t)$ is decreasing in $t,$ we have the double inequality
$$
\frac{r^2}{|a|^2(|a|+r)^2}=
h(r/|a|)\le |x'-a'|^2\le h(-r/|a|)=\frac{r^2}{|a|^2(|a|-r)^2},
$$
which is equivalent to
$$
\frac{r}{|a|(|a|+r)}\le |x'-a'| \le \frac{r}{|a|(|a|-r)}.
$$
In view of the above estimates, we get the inclusion relation $\A(a'; R_0, R_1)\subset \A',$
where
$$
R_0=\frac{r_0}{|a|(|a|-r_0)}
\aand
R_1=\frac{r_1}{|a|(|a|+r_1)}.
$$
Put $m=\mo\A$ so that $r_1=e^mr_0.$
Since $r_1\le|a|,$ the range of $r_0$ is $0<r_0\le e^{-m}|a|.$
Hence
$$
\frac{R_1}{R_0}
=\frac{r_1(|a|-r_0)}{r_0(|a|+e^mr_0)}
=\frac{e^m(|a|-r_0)}{|a|+e^mr_0}
\ge\frac{e^m(|a|-e^{-m}|a|)}{|a|+|a|}
=\frac{e^{m}-1}2.
$$
By the condition $m>\log 3,$ we see that the right-most term is greater than 1.
Hence, $\A_0=\A(a'; R_0, R_1)$ is an annular subring of $\A'$ with
$$
\mo\A_0=\log\frac{R_1}{R_0}\ge\log\frac{e^{m}-1}2
=m+\log\frac{1-e^{-m}}2
\ge m-\log 3.
$$
Next we consider the case when $0\in C_0;$ namely, $|a|\le r_0.$
For a point $x$ on the sphere $|x-a|=r$ with $r>|a|,$ we denote by $x'$ the inversion $I(x).$
Since $|x'|=1/|x|,$ by the triangle inequality, we have
$$
\frac1{r+|a|}\le |x'|\le \frac1{r-|a|}.
$$
Thus the ring $\A'=I(\A)$ contains the annular ring $\A(0; R_0,R_1)$ as a subring,
where
$$
R_0=\frac1{r_1-|a|}\aand R_1=\frac1{r_0+|a|}.
$$
With the relation $r_1=e^m r_0,$ we estimate
$$
\frac{R_1}{R_0}=\frac{e^mr_0-|a|}{r_0+|a|}
\ge \frac{e^m|a|-|a|}{|a|+|a|}=\frac{e^m-1}2.
$$
Thus, as in the previous case, we see that $\A_0=\A(0; R_0,R_1)$ is a subring of $\A'$
with $\mo\A_0\ge \mo\A-\log 3.$
\end{proof}

\subsection{Proof of Theorem \ref{thm:sep0}}
Suppose that $E$ satisfies \eqref{eq:UP} for a constant $c\in(0,1).$
We first assume that $\infty\in E$ so that $\diam E=\infty.$
Let $\ring=\ring(C_0,C_1)$ be a ring separating $E$ with $\mo\ring\ge A_n-\log c.$
Choose a point $a$ from the bounded component $C_0.$
Then $\ring$ separates $a$ from $\infty$ thus Theorem \ref{thm:teich} implies
that there exists an annular subring $\A=\A(a; r_0, r_1)$ of $\ring$ with
$\mo\A=\log(r_1/r_0)\ge\mo\ring-A_n.$
Since $r_0\le cr_1$ by assumption, we have
$$
\{x\in E: cr_1< |x-a|<r_1\}~\subset~ E\cap\A
~\subset~ E\cap\ring=\varnothing,
$$
which contradicts \eqref{eq:UP}.
Hence we conclude that $\mo\ring\le A_n-\log c$ for a ring $\ring$ separating $E$ when $\infty\in E.$
We next assume that $E\subset\R^n.$
Let $\ring=\ring(C_0,C_1)$ be a ring separating $E$ with $\mo\ring\ge A_n+\log(3/c)$
and choose a point $a$ from the bounded component $C_0.$
Set $\ring'=I_a(\ring)\subset\R^n,$ where $I_a(x)=(x-a)/|x-a|^2+a.$
Then, by Theorem \ref{thm:teich}, $\ring'$ contains an annular subring $\A$ centered at $a$
with $\mo\A\ge \mo\ring'-A_n=\mo\ring-A_n.$
By Lemma \ref{lem:ann}, we find an annular subring $\A_0=\A(a; r_0, r_1)$ of $I_a(\A')$
centered at $a$ such that $\mo\A_0\ge \A'-\log 3\ge -\log c.$
We now see that $\{x\in E: cr_1< |x-a|<r_1\}\subset E\cap\A_0
\subset E\cap\ring=\varnothing$ as in the first case.
Since $a\in E$ and $\A_0$ separates $E,$ we have $\diam E\ge r_1.$
This contradicts \eqref{eq:UP}.

Conversely, we suppose that there exists a constant $M>0$ such that
$\mo\ring\le M$ whenever a ring $\ring$ separates $E.$
We show that \eqref{eq:UP} is valid for $c=\min\{e^{-M}, 1/2\}.$
Indeed, to the contrary, we assume that $\{x\in E: cr< |x-a|<r\}$
is empty for some $a\in E,~ a\ne\infty,$ and $0<r<\diam E.$
If the annular ring $\A=\A(a; cr,r)$ separates $E,$ we would have
$\mo\A=\log(1/c)\le M$ by assumption.
However, this is impossible by the choice of $c.$
Thus $\A$ cannot separate $E.$
This implies that $E$ is contained in the bounded component $|x-a|\le cr$
and, in particular, $\diam E\le 2cr\le r,$ which contradicts the choice of $r.$
Now the proof is complete.\QED

\subsection{Another consequence of Theorem \ref{thm:teich}}
Fix a number $B$ so that $B>A_n,$ where $A_n$ is given in \S \ref{ssec:An}.
Let $\ring=\ring(C_0,C_1)$ be a ring in $\R^n$ with $m=\mo \ring\ge B~(>A_n).$
By Theorem \ref{thm:teich}, there is an annular subring $\A=\A(a; r_0,r_1)$  of $\ring$
with $\mo\A\ge\mo\ring-A_n.$
Then we easily get $\diam C_0\le 2r_0$ and $\dist(C_0,C_1)\ge r_1-r_0.$
Here and hereafter, $\dist(C_0,C_1)=\inf\{|x_0-x_1|: x_0\in C_0, x_1\in C_1\}$
denotes the Euclidean distance between $C_0$ and $C_1.$
Since $r_1/r_0=e^{\mo\A}\ge e^{m-A_n},$ we get
$$
r_1-r_0=r_0\big(e^{m-A_n}-1\big)\ge r_0e^m\big(e^{-A_n}-e^{-B}\big).
$$
These observations yield the following corollary.

\begin{cor}\label{cor:C}
Let $B>A_n$ and $\ring=\ring(C_0,C_1)$ be a ring in $\R^n$ with $\infty\in C_1$
and $\mo\ring\ge B.$
Then
\begin{equation*}
\diam C_0\le Me^{-\mo\ring}\dist(C_0,C_1),
\end{equation*}
where $M$ is the constant $2/(e^{-A_n}-e^{-B}).$
\end{cor}

We remark that a similar result was obtained in \cite{GMSV05} for the planar case.

\section{Boundary correspondence}

In this section, we consider the problem  when a given
homeomorphism $f$ of the unit ball $\B^n$ onto itself extends to the boundary homeomorphically.
Gutlyanski\u\i, Sakan and the second author formulated in \cite{GSS13} a necessary and sufficient
condition for such an $f$ to extend homeomorphically to the boundary
in terms of the moduli of semiannuli in the case when $n=2.$
We extend it to higher dimensional cases.
Note that some results below are straightforward extensions of the two-dimensional case in \cite{GSS13}
but the proofs need some more efforts because conformal mappings in higher dimensions are only
M\"obius transformations.

\subsection{Semirings}
Our standard model for ``semiring" is the upper half of the {\it closed} ring
$$
\T_R=\{x\in\Hn: 1\le |x|\le R\}
$$
for $1<R<+\infty.$
Here $\Hn$ denotes the upper half space $\{x=(x_1,\dots,x_n): x_n>0\}.$
The semiring $\T_R$ has the distinguished boundary components
$\partial_0\T_R=\{x\in\Hn: |x|=1\}$ and
$\partial_1\T_R=\{x\in\Hn: |x|=R\}$ relative to  $\Hn\,,$ which are homeomorphic to
the $(n-1)$-dimensional open ball $\B^{n-1}.$
Let $\Gamma(R)$ denote the family of arcs $\gamma:[0,1]\to \T_R$
joining $\partial_0\T_R$ and $\partial_1\T_R$ in $\T_R.$
Thanks to \cite[7.7]{Vai71}, we obtain the formula
\begin{equation}\label{eq:modulus}
\M(\Gamma(R))=\frac{\omega_{n-1}}{2}\left(\log R\right)^{1-n}.
\end{equation}
A subset $\es$ of $\Sn$ is called a semiring if it is homeomorphic to $\T_R$ for some $R>1.$
We denote by $\Gamma_\es$ the family of the image curves of $\Gamma(R)$
under a homeomorphism $f:\T_R\to \es.$
Note that $\Gamma_\es$ does not depend on the particular choice of $f$ and $R.$
We define the modulus of the semiring $\es$ by
$$
\mo \es=\left[\frac{\omega_{n-1}}{2\M(\Gamma_\es)}\right]^{1/(n-1)}.
$$
We have the formula $\mo\T_R=\log R$ by virtue of \eqref{eq:modulus}.
Let $G$ be a proper subdomain of $\Sn.$
A semiring $\es$ in $G$ is said to be {\it properly embedded} in $G$ if
$\es\cap C$ is compact whenever $C$ is a compact subset of $G.$
That is to say, $\es$ is a properly embedded semiring in $G$ if and only if
for some (and thus for every) homeomorphism $f:\T_R\to S$ is proper as considered to be
a map $f:\T_R\to G.$
Note that $\partial_0\es=f(\partial_0\T_R)$ and $\partial_1\es=f(\partial_1\T_R)$
are properly embedded $(n-1)$-balls in $G$ and constitute connected components
of $\partial \es\cap G.$
(Though there is no canonical way to label $\partial_0\es$ and $\partial_1\es$
to the connected components of $\partial\es$ in $G,$ we take the labels given by
a proper embedding $f:\T_R\to G$ and fix them for convenience.)

From now on, we consider a semiring properly embedded in $\B^n$ by a mapping
$f:\T_R\to\es\subset\B^n.$
Then $\B^n\setminus\es$ is an open subset of $\B^n$ consisting of two components $V_0$ and $V_1$
for which $V_0\cap\partial_1\es=\emptyset$ and $V_1\cap\partial_0\es=\emptyset.$

Our main tool is the following separation lemma.
The planar case was given in \cite{GSS13}.

\begin{lem}\label{lem:sep}
Let $\es$ be a properly embedded semiring in $\B^n.$
Then $\mo \es>0$ if and only if the Euclidean distance $\delta=\dist(V_0, V_1)$ between
$V_0$ and $V_1$ is positive.
Moreover, in this case, the double $\hat\es:=\Int\es\cup U \cup \Int\es^*$ of $\es$
is a ring with
$\mo\hat\es=\mo\es,$ where $\es^*$ is the reflection of $\es$ in $\partial\B^n$
and $U=\partial\B^n\setminus(\overline V_0\cup\overline V_1).$
\end{lem}

\begin{proof}
We recall that $\mo\es>0$ if and only if $\M(\Gamma_\es)<+\infty.$
Assume first that $\delta>0.$
In this case, the function $\rho_0=\chi_{\B^n}/\delta$ is admissible for the curve family
$\Gamma_\es.$
Therefore, we have
$$
\M(\Gamma_\es)\le \int \rho_0^n\,dm=\frac{{\rm Vol}(\B^n)}{\delta^n}<+\infty.
$$
Assume next that $\delta=0.$
Then there is a point $x_0$ in the set $\overline V_0\cap\overline V_1~(\subset\partial\B^n).$
Since $V_0$ and $V_1$ are both continua, the sphere $|x-x_0|=t$ intersects
both of $V_0$ and $V_1$ for small enough $t>0.$
Therefore, Theorem~10.12 in V\"ais\"al\"a \cite{Vai71} implies that
$\M(\Gamma_\es)=+\infty.$

Suppose $\mod\es>0.$
Then $\delta>0$ and $\hat V_j=\overline V_j\cup V_j^*~(j=1,2)$ are disjoint continua.
Obviously, $\hat\es=\Sn\setminus(\hat V_0\cup\hat V_1)$ and thus $\hat\es$ is a ring.
The equality $\mo\hat\es=\mo\es$ follows from the symmetry principle for the moduli
of curve families (see Theorem 4.3.3 or its corollary in \cite{GMP17}).
\end{proof}

\subsection{Canonical semirings in $\B^n$}
For a point $\xi\in\partial\B^n$ and real numbers $0<r_0<r_1<+\infty,$
we consider the properly embedded semiring
$$
\T(\xi; r_0, r_1)=\left\{x\in\B^n: r_0\le \frac{|x-\xi|}{|x+\xi|}\le r_1\right\}
$$
in $\B^n.$

\begin{lem}\label{lem:T}
$\displaystyle
\mo\T(\xi; r_0, r_1)=\log\frac{r_1}{r_0}.
$
\end{lem}

\begin{proof}
Let $Q:\Sn\to\Sn$ be the reflection in the sphere
$x_1^2+\cdots+x_{n-1}^2+(x_n-1)^2=2,$ in other words,
$$
Q(x)=2\frac{x-e_n}{|x-e_n|^2}+e_n,
$$
where $e_n=(0,\dots, 0, 1)\in\R^n,$ and let $P$ be the reflection
in the hyperplane $x_n=0,$ namely, $P(x)=x-2(x\cdot e_n)e_n.$
Then the M\"obius transformation $P\circ Q$ is known as
the stereographic projection which maps $\B^n$ onto $\Hn$
and $\partial\B^n\setminus\{e_n\}$ onto $\partial\Hn\setminus\{\infty\}
=\R^{n-1}\times\{0\},$ respectively.
Choose a rotation $R:\R^n\to\R^n$ about the origin so that
$R(\xi)=e_n$ and $R(-\xi)=-e_n$ and set $M=M_\xi=P\circ Q\circ R:\B^n\to\Hn.$
Put $y=R(x)$ and $z=Q(y)$ for $x\in\B^n.$
Then $|x-\xi|=|y-e_n|$ and $|x+\xi|=|y+e_n|.$
Moreover, since $|y-e_n|^2z=2(y-e_n)+|y-e_n|^2e_n,$ we compute
\begin{align*}
|y-e_n|^4|z|^2&=4|y-e_n|^2+4|y-e_n|^2(y-e_n)\cdot e_n+|y-e_n|^4 \\
&=|y-e_n|^2\big[4+4y\cdot e_n-4+|y|^2-2y\cdot e_n+1\big] \\
&=|y-e_n|^2|y+e_n|^2
\end{align*}
and thus
$$
|M(x)|=|z|=\frac{|y+e_n|}{|y-e_n|}=\frac{|x+\xi|}{|x-\xi|}.
$$
In this way, we see that the M\"obius transformation $M$ maps the set
$\T(\xi;r_0,r_1)$ onto the semiannulus $\{x\in\Hn: 1/r_1\le |x|\le 1/r_0\},$
whose modulus is equal to $\log(r_1/r_0)$.
Since the modulus is conformally invariant, the required formula follows.
\end{proof}

The unit ball $\B^n$ carries the hyperbolic distance $h(x_1,x_2)$
induced by the hyperbolic metric $2|dx|/(1-|x|^2)$
so that we may develop hyperbolic geometry on $\B^n.$
See \cite{Be} for details.

The following lemma was shown in \cite{GSS13} for the 2-dimensional case.

\begin{lem}\label{lem:diam}
Let $\T$ be a properly embedded semiannulus in $\B^n$ whose boundary in $\B^n$
consists of two hyperbolic hyperplanes and let $W_0$ and $W_1$ be the
connected components of $\B^n\setminus\T.$ Then the Euclidean diameters of
$W_0$ and $W_1$ satisfy the inequality
$$
\min\{\diam W_0, \diam W_1\}\le\frac 2{\cosh(\frac12\mo \T)}.
$$
Equality holds if and only if $\T$ is of the form $\T(\xi; r,1/r)$
for some $\xi\in\partial\B^n$ and $0<r<1.$
\end{lem}

\begin{proof}
Fixing the value of $\mo\T,$ we shall find the configuration of $W_0$ and $W_1$ for which
$\min\{\diam W_0, \diam W_1\}$ is maximized.
Let $H_j=\partial W_j\cap\B^n$ for $j=0,1.$
There is a unique hyperbolic line $l$ in $\B^n$ which is perpendicular to both of $H_0$
and $H_1.$
The hyperbolic length $\delta$ of $l\cap\es$ is nothing but the hyperbolic distance
between $W_0$ and $W_1.$
Since $l$ is a part of a circle (possibly a line) intersecting $\partial\B^n$ perpendicularly,
there is a (two-dimensional) plane $\Pi$ containing $l$ and the origin.
Note that $\diam W_j\cap\Pi =\diam W_j$ for $j=0,1.$
Since $\B^n\cap\Pi$ is (hyperbolically) isometric $\B^2,$
the problem now reduces to the two-dimensional case.
Hence the inequality in the assertion follows from \cite[Lemma 2.6]{GSS13}%
\footnote{Note that the definition of the hyperbolic metric is different in \cite{GSS13}
from here by the factor 2.}.
Equality holds only if $l\cap\es$ is the line segment with the origin as its midpoint.
In this case, $\T=\T(\xi; r,1/r),$ where $\xi$ is one of the end points of $l$
and $r=\tanh(\delta/4).$
\end{proof}

\subsection{Separation theorem}
The following result is a generalization of Theorem~2.3 in \cite{GSS13}.
We note that we lose half of the modulus in the exponent though the previous lemma is sharp.

\begin{thm}\label{thm:sep}
Let $\es$ be a properly embedded semiannulus in $\B^n.$
Then the connected components $V_0$ and $V_1$ of $\B^n\setminus\es$ satisfy the inequality
\begin{equation}\label{eq:sep}
\min\{\diam V_0,\diam V_1\}\le Q_n \exp\left(-\frac12\mo\es\right),
\end{equation}
where $Q_n=4\exp(A_n/2).$
\end{thm}

\begin{proof}
When $\mo\es\le 2\log(Q_n/2),$ the right-hand side of \eqref{eq:sep}
is at least $2.$
Therefore, \eqref{eq:sep} trivially holds.
We now assume that $\mo\es>2\log(Q_n/2)=A_n+2\log 2.$
By Lemma \ref{lem:sep}, the extended set $\hat\es$ is
a ring with $\mo\hat\es=\mo\es>A_n.$
Choose a point $\xi_j$ from $\overline V_j\cap\partial\B^n$ for each $j=0,1$ and
consider the M\"obius mapping $L:\Sn\to\Sn$ defined by $L(x)=M_{\xi_0}(x)-M_{\xi_0}(\xi_1),$
where $M_\xi$ is constructed in the proof of Lemma \ref{lem:T}.
By definition, $L(\B^n)=\Hn,~ L(\xi_0)=\infty$ and $L(\xi_1)=0.$
In particular, $L(\hat\es)$ separates $0$ from $\infty.$
Theorem \ref{thm:teich} now yields an annular subring $\A=\{x: r_0<|x|<r_1\}$ of $L(\hat\es)$
for some $0<r_0<r_1<+\infty$ with $\mo\A=\log(r_1/r_0)\ge \mo\es-A_n.$
We set $\T=L^{-1}(\A\cap\Hn)$ and let $W_j$ be the connected component
of $\B^n\setminus \T$ containing $V_j$ for $j=0,1.$
Then $\mo\T=\mo\A$ and Lemma \ref{lem:diam} now implies the inequalities
\begin{align*}
\min\{\diam V_0, \diam V_1\}
&\le\min\{\diam W_0, \diam W_1\} \\
&\le\frac 2{\cosh(\frac12\mo \T)} \\
&\le 4\exp\left(-\frac12\mo \T\right) \\
&\le 4\exp\left(-\frac12\mo\es+\frac12 A_n\right).
\end{align*}
Thus the assertion follows.
\end{proof}

\section{Applications to quasiconformal maps}
\subsection{Quasiconformal maps}
Modulus estimates are powerful tools to deal with general homeomorphisms
of domains such as solutions to degenerate Beltrami equations
(see, for instance, \cite{GMSV05} or \cite{GSS13}).
In this section, for simplicity,
we give several applications of the results presented above
to quasiconformal mappings.
More applications will be presented in our forthcoming paper.

For a definition and basic properties of quasiconformal maps,
we refer to V\"ais\"al\"a's book \cite{Vai71}
and a recent monograph \cite{GMP17}.
The most important property of quasiconformal mappings in our context is
quasi-invariance for the moduli of curve families.
That is to say, for a $K$-quasiconformal homeomorphism $f:G\to G'$ between
domains in $\Sn,$ we have the double inequality $K^{-1}\M(\Gamma)\le
\M(f(\Gamma))\le K\, \M(\Gamma)$
for all curve families $\Gamma$ in $G.$
Note also the following fact: a homeomorphism $f:G\to G'$ is $K$-quasiconformal
if and only if the double inequality $K^{-1}\M(\Gamma_\ring)\le \M(\Gamma_{f(\ring)})
\le K\M(\Gamma_\ring),$ equivalently
$$
K^{-1/(n-1)}\mo\ring\le\mo f(\ring)\le K^{1/(n-1)}\mo\ring,
$$
holds for every ring $\ring$ whose closure is contained in $G$ (see \cite[Cor.~36.2]{Vai71}).

\subsection{Conditions for continuity at boundary}
The next proposition is an extension of \cite[Prop.~3.1]{GSS13}
and it is easily verified by Theorem \ref{thm:sep}.

\begin{prop}\label{prop:lim}
Let $f:\B^n\to\B^n$ be a homeomorphism and $\xi\in\partial\B^n.$
The mapping $f$ extends continuously to the point $\xi$ if
$$
\lim_{r\to0+}\mo f(\T(\xi; r,R))=+\infty
$$
for some $R>0.$
\end{prop}

\begin{proof}
Let $\es_r=f(\T(\xi; r,R))$ and denote by $V_0(r)$ and $V_1$
the images of the sets $\{x\in\B^n: |x-\xi|/|x+\xi|<r\}$ and
$\{x\in\B^n: |x-\xi|/|x+\xi|>r\}$ under the mapping $f,$ respectively.
Since $\mod\es_r\to+\infty,$
Theorem \ref{thm:sep} implies that $\diam V_0(r)\to0$ as $r\to0.$
Therefore, the cluster set $\bigcap_{0<r<R}\overline{V_0(r)}$
consists of one point, to which $f(x)$ converges as $x\to\xi$ in $\B^n.$
\end{proof}

If we had more precise information on the rate of convergence
of $\mo f(\T(\xi;r,R)),$ we could get an estimate of modulus of continuity
of $f(x)$ at the boundary point $\xi.$
We also have the following theorem.

\begin{thm}
Let $E$ be a subset of $\partial\B^n$ and $f:\B^n\to\B^n$
be a homeomorphism. Suppose further that for every $\xi\in E,$
$$
\lim_{r\to0+}\mo f(\T(\xi; r,R))=+\infty
$$
holds for some number $R=R_\xi>0.$
Then $f$ extends to a continuous injection of $\B^n\cup E$ into $\overline\B^n.$
\end{thm}

\begin{proof}
By the above proposition, we see that $f$ extends continuously to the set $E$
so that $f(E)\subset\partial\B^n.$
We show that the extended map $f$ is injective on $E.$
Suppose, to the contrary, that $f(\xi_1)=f(\xi_2)=:\omega_0$ for some $\xi_1, \xi_2
\in E$ with $\xi_1\ne\xi_2.$
We take $R>0$ so small that $\T(\xi_1; r,R)\cap\T(\xi_2; r,R)=\emptyset$
for $0<r<R.$
Let $V_0, V_1$ be the connected components of $\B^n\setminus\T$ with $\xi_1\in\overline V_0,$
where $\T=\T(\xi_1; r,R).$
Take two sequences $z_k, z_k'\in\B^n~(k=1,2,3,\dots)$ so that $z_k\to\xi_1$
and $z_k'\to\xi_2.$
Then $z_k\in V_0$ and $z_k'\in V_1$ for sufficiently large $k.$
In particular, $\dist(f(V_0), f(V_1))\le |f(z_k)-f(z_k')|$ for such a $k.$
Since $f(z_k)\to\omega_0$ and $f(z_k')\to\omega_0$ as $k\to\infty,$
we have $\dist(f(V_0),f(V_1))=0.$
By Lemma \ref{lem:sep}, we conclude that $\mo f(\T)=0,$ which contradicts
the assumption that $\mo f(\T(\xi_1; r,R))\to+\infty$ as $r\to0^+.$
\end{proof}

Letting $E=\partial\B^n,$ we obtain the following result.
(For the case when $n=2,$ see \cite{Bra07}, \cite[Cor.~3.3]{GSS13}.)

\begin{thm}\label{thm:exten}
A homeomorphism $f:\B^n\to\B^n$ extends to a homeomorphism $f:\overline\B^n\to
\overline\B^n$ if and only if for each $\xi\in\partial\B^n,$ there is an $R=R_\xi>0$ such that
$$
\lim_{r\to0+}\mo f(\T(\xi; r,R))=+\infty.
$$
\end{thm}

\subsection{Boundary extension of quasiconformal maps of the unit ball}
It is well known that a quasiconformal automorphism of $\B^n$ extends
to the boundary homeomorphically.
See Section 17 of \cite{Vai71} for more information on this topic.
Here is a version of such a theorem.

\begin{thm}\label{thm:Bn}
Let $f:\B^n\to\B^n$ be a $K$-quasiconformal mapping fixing the origin.
Then $f$ extends to a homeomorphism $\tilde f:\overline\B^n\to\overline\B^n$ so that
$$
|f(x)-\tilde f(\xi)|\le C(n) |x-\xi|^{\alpha/2},\quad x\in\B^n, \xi\in\partial\B^n.
$$
Here $\alpha=1/K^{1/(n-1)}$ and $C(n)$ is a constant depending only on $n.$
\end{thm}

Indeed, the much better estimate $|f(x)-f(y)|\le 4\lambda_n^2|x-y|^{\alpha}$
for $x,y\in\B^n$ 
is known
(see \cite[Theorem 6.6.1]{GMP17}) where $\lambda_n$ is the Gr\"otzsch ring constant.  Moreover,
$\lambda_n$ can also be replaced by a constant independent of the dimension $n,$ see
\cite{AV88}.
Here, we give a proof of the above result as a simple application of our Theorem \ref{thm:sep}.

\begin{proof}
Let $\T=\T(\xi;r,1)$ for $\xi\in\B^n,~0<r<1$ and $\es=f(\T).$
Note that $\M(\Gamma_\es)\le K\,\M(\Gamma_\T)$ and thus
$\mo\es\ge K^{-1/(n-1)}\mo\T=K^{-1/(n-1)}\log(1/r)$ by Lemma \ref{lem:T}.
In particular, we see that $\mo f(\T(\xi;r,1))\to+\infty$ as $r\to0$ for each
$\xi\in\partial\B^n.$
Hence Theorem \ref{thm:exten} guarantees that $f$ extends to a homeomorphism
$\tilde f$ of $\overline\B^n.$
Fix $\xi\in\partial\B^n$
and consider the ring $\es=f(\T(\xi; \ep,1))$ properly embedded in $\B^n,$
Theorem \ref{thm:sep} now yields the following inequalities
$$
\min_{j=0,1}\diam V_j\le Q_n \exp\left(-\frac12\mo\es\right)
\le Q_n\exp\bigg( \frac{\log\ep}{2K^{1/(n-1)}} \bigg)
=Q_n\exp\left(\frac\alpha2\log\ep\right),
$$
where $V_1$ is the component of $\B^n\setminus\es$ satisfying $0\in\partial V_1$
and $V_0$ is the other one.
Note that $\tilde f(\xi)\in\partial V_0$ and that $\diam V_1\ge 1.$
If $(\alpha/2)\log\ep<-\log Q_n,$ the right-most term in the above inequalities
is less than 1, which implies
$$
\diam V_0\le Q_n\exp\big(\tfrac\alpha2\log\ep) \big)=Q_n \ep^{\alpha/2}.
$$
We now put $\ep_0=\exp(-(2/\alpha)\log Q_n)=Q_n^{-2/\alpha}.$
If $|x-\xi|<\ep_0,$ we have $|x-\xi|/|x+\xi|\le|x-\xi|/(2-|x-\xi|)<|x-\xi|<\ep_0.$
Letting $\ep=|x-\xi|,$ we obtain
$$
|f(x)-\tilde f(\xi)|\le\diam V_0\le Q_n|x-\xi|^{\alpha/2}.
$$
If $|x-\xi|\ge\ep_0,$ we make the trivial estimates
$$
|f(x)-\tilde f(\xi)|\le 2\le 2\left(\frac{|x-\xi|}{\ep_0}\right)^{\alpha/2}
=2Q_n|x-\xi|^{\alpha/2}.
$$
Thus we see that $C(n)=2Q_n$ works.
\end{proof}

\subsection{Boundary extension of quasiconformal maps of the half space}
In the case of the unit ball, the optimal H\"older exponent is known to be $1/K^{1/(n-1)}.$
In the assertion of Theorem \ref{thm:Bn}, however, 
we have the extra factor 2.
If we do not care about uniformity of the estimate, we can get rid of it.
For instance, in the case of half space $\Hn,$ we have a similar result
with an optimal exponent.

\begin{thm}\label{thm:Hn}
Let $f:\Hn\to\Hn$ be a $K$-quasiconformal homeomorphism fixing $e_n=(0,\dots,0,1).$
Suppose that $f(x)\to\infty$ as $x\to\infty$ in $\Hn.$
Then $f$ extends to a homeomorphism $\tilde f:\overline{\mathbb{H}}^n
\to\overline{\mathbb{H}}^n$ and,
for every $R>0,$ there exists a constant $C=C(R,K,n)>0$ such that
$$
|f(x)-\tilde f(\xi)|\le C|x-\xi|^\alpha
$$
whenever $\xi\in\partial\Hn,~ |\xi|\le R$ and $x\in\Hn,~|x-\xi|\le 1.$
Here $\alpha=1/K^{1/(n-1)}.$
\end{thm}

For the proof, we first prepare an estimate for $K$-quasiconformal automorphisms of $\Hn$
fixing the basepoint $e_n.$
For $0<r<1,$ we set $B(r)=\{x\in\Hn: |x-e_n|\le r|x+e_n|\}.$
Note that $B(r)$ is the closed ball centered at $\frac{1+r^2}{1-r^2}e_n$
with radius $\frac{2r}{1-r^2}$ (the so-called Apollonian ball).
Also note that $\bigcup_{0<r<1}B(r)=\Hn.$
The next result is a variant of the well-known quasiconformal Schwarz Lemma
and the function $\varphi_{K,n}(r)$ is known as the distortion function
(see \cite[Ch. 8]{AVV97}).

\begin{lem}\label{lem:dst}
Let $f:\Hn\to\Hn$ be a $K$-quasiconformal map fixing $e_n.$
Then $f$ maps $B(r)$ into $B(r'),$ where $r'=\varphi_{K,n}(r)
:=1/\gamma_n^{-1}(K\gamma_n(1/r)).$
\end{lem}

\begin{proof}
Let $M=M_{-e_n}$ and $g=M^{-1}\circ f\circ M:\Sn\setminus\overline\B^n
\to\Sn\setminus\overline\B^n,$ where $M_\xi$ is given in the proof of Lemma \ref{lem:T}.
Note here that $M^{-1}(e_n)=\infty.$
For a fixed $x\in\Hn,$ we put $y=M^{-1}(x).$
Consider the ring $\ring=\ring(\overline\B^n,[y,\infty]),$ which is a rotation
of the Gr\"otzsch ring $R_{G,n}(|y|)$ about the origin.
Then the image $\ring'=g(\ring)$ is a ring separating $y'=g(y)$ and $\infty$ from $\overline\B^n.$
Lemma \ref{lem:GR} now implies the inequality $\M(\Gamma_{\ring'})\ge\gamma_n(|y'|).$
On the other hand, $K$-quasiconformality of $g$ implies
$\M(\Gamma_{\ring'})\le K\,\M(\Gamma_{\ring})=K\gamma_n(|y|).$
Hence, $\gamma_n(|y'|)\le K\gamma_n(|y|),$ equivalently,
$|y'|\ge \gamma_n^{-1}(K\gamma_n(|y|)).$
Note that $x\in B(r)$ precisely if $|y|\ge 1/r.$
Thus the assertion follows.
\end{proof}

\subsection{Proof of Theorem \ref{thm:Hn}}
Note that $f$ can be extended to a homeomorphism $\tilde f$ of $\overline{\mathbb{H}}^n$
by Theorem \ref{thm:Bn}, because $\Hn$ and $\B^n$ are M\"obius equivalent.

We first show the claim that there is a constant $\rho=\rho(R,K,n)>0$ such that
$|f(x)|\le \rho$ for all $x\in\Hn$ with $|x|\le 1+R.$
Let $\es=\{x\in\Hn: 1+R\le |x|\le R'\},$ where $R'$ is chosen so that
$\log [R'/(1+R)]=(A_n+\log 2)/\alpha,$ where $A_n$ appears in \S \ref{ssec:An}.
Then, as in the proof of Theorem \ref{thm:Bn}, we have
$$
\mo\es'\ge K^{-1/(n-1)}\mo\es=\alpha\log \frac{R'}{1+R}=A_n+\log 2>0.
$$
We may reflect $\es'=f(\es)$ in the hyperplane $x_n=0$ as before
to obtain a ring $\hat\es'=\ring(C_0, C_1)$ with $\mo\hat\es'=\mo\es'.$
Since $\mo\hat\es'>A_n,$ Theorem \ref{thm:teich} yields an annular
subring $\A$ of $\hat\es'$ of the form $r_0<|y-a|<r_1,$ where
$\log(r_1/r_0)\ge \mo\es'-A_n\ge \log 2$ and $a=\tilde f(0)\in C_0.$
In view of the fact that $e_n=f(e_n)\in C_0,$ we have $|e_n-a|\le r_0$
and thus $|a|\le 1+r_0.$
Noting $w=f(R'e_n)\in C_1,$ we have $|w-a|\ge r_1.$
On the other hand, since $R'e_n\in B(r)$ with $r=(R'-1)/(R'+1),$
Lemma \ref{lem:dst} implies $w\in B(r')$ for $r'=\varphi_{K,n}(r).$
In particular, we obtain $|w|\le (1+r')/(1-r')$ and thus
$$
r_1\le |w-a|\le |w|+|a|\le \frac{1+r'}{1-r'}+1+r_0.
$$
Noting the inequality $2r_0\le r_1,$ we finally have $r_0\le 2/(1-r').$
Since the set $C_0=f(\{x\in\Hn: |x|\le R\})$ is contained in the ball
$|y-a|\le r_0,$ the claim follows with $\rho=1+4/(1-r')$.

By the last claim, we have $|f(x)|, |\tilde f(\xi)|\le\rho$ for $\xi\in\partial\Hn$ with $|\xi|\le R$
and $x\in\Hn$ with $|x-\xi|\le 1.$
For such a point $\xi$ we consider the semiring $\es=\es(\xi;r,1)=\{x\in\Hn:r\le |x-\xi|\le 1\}$
for $0<r<\delta$ and its image $\es'=f(\es),$ where $\delta$ is determined by the relation
$-\alpha\log \delta=A_n+\log 2.$
Then, as above, we have 
$$
\mo \es'\ge K^{-1/(n-1)}\mo\es=\alpha\log\frac1r>A_n+\log 2.
$$
Let $\hat\es$ be the double of $\es.$
An application of Corollary \ref{cor:C} with $B=A_n+\log 2$ to $\hat\es'=\ring(C_0,C_1)$
gives us the estimate
$$
\diam C_0\le Me^{-\mo\es'}\dist(C_0,C_1)
\le Mr^\alpha\dist(C_0,C_1),
\quad M=4e^{A_n}.
$$
Since $\tilde f(\xi)\in C_0$ and $f(e_n)=e_n\in C_1,$
we have $\dist(C_0,C_1)\le |\tilde f(\xi)-e_n|\le \rho+1.$
Therefore, if $|x-\xi|=r<\delta,$ we obtain
$$
|f(x)-\tilde f(\xi)|\le \diam C_0\le (\rho+1)M|x-\xi|^\alpha.
$$
If $|x-\xi|\ge \delta,$ we have
$$
|f(x)-\tilde f(\xi)|\le 2\rho\le 2\rho\left(\frac{|x-\xi|}{\delta}\right)^\alpha
=\rho M|x-\xi|^\alpha.
$$
Hence we obtain the required inequality with $C=(\rho+1)M.$\QED

\noindent
{\bf Acknowledgements.}
The sudden death of Prof. Stephan Ruscheweyh, a leading researcher in geometric function theory,
has left an irreplaceable gap in the research community. His vision was to establish forums for
international meeting of colleagues from different corners of the world and for presentation of latest research  ideas, 
by organizing a series of the CMFT conferences and by founding the CMFT journal. 
His ideas to apply computational methods to function theoretic problems will continue 
to inspire the research of geometric function theory for the many years to come.
The authors would like to take this opportunity to express their sincere appreciation for
his lifetime achievements.

Finally, the authors would also like to thank the referee for careful checking the
manuscript and suggestions.



\begin{thebibliography} {99}

\bibitem{ABP}
{\sc
V.~Alberge, M.~Brakalova-Trevithick and A.~Papadopoulos,}
{\em A Commentary on Teichm\"uller's paper ``Untersuchungen \"uber konforme und quasikonforme Abbildungen" },
in {\em Handbook of Teichm\"uller Theory} (A.~Papadopoulos, ed.), Vol.~VII, pp.~561--584,
EMS Publishing House, Z\"urich, 2020.

\bibitem {AV88}
{\sc
G.~D.~Anderson, M.~K.~Vamanamurthy,}  H\"older continuity of quasiconformal mappings of the unit ball. Proc. Amer. Math. Soc. 104 (1988), no. 1, 227--230.

\bibitem {AVV97}
{\sc
G.~D.~Anderson and M.~K.~Vamanamurthy and M.~K.~Vuorinen,}  {\em Conformal invariants, inequalities, and quasiconformal maps}, John Wiley \& Sons, Inc., New York, 1997.

\bibitem {AW09}
{\sc
F.~G.~Avkhadiev and K.-J.~Wirths,}  {\em Schwarz-Pick type inequalities},
Frontiers in Mathematics. Birkh\"auser Verlag, Basel, 2009.

\bibitem{Be}
{\sc A.~F.~Beardon,} 
{\em The geometry of discrete groups}, Graduate Texts in Math., Vol. 91, Springer-Verlag, New York, 1983.

\bibitem{Bet03}
{\sc D.~Betsakos,}  \emph{On separating conformal annuli and Mori's ring domain in $\R^n$}, Israel J. Math. \textbf{133} (2003), 1--8.

\bibitem{Bra07}
{\sc M.~A. Brakalova,}  \emph{Boundary extension of {$\mu$}-homeomorphisms}, Complex
  and {H}armonic {A}nalysis, DEStech Publ., Inc., Lancaster, PA, 2007,
  pp.~231--247.


\bibitem{gm} { \textsc{J. B. Garnett and D. E. Marshall,}}  {\em Harmonic measure. }
Cambridge Univ. Press 2005, xvi+571pp.




\bibitem {GMP17}
{\sc F.~W.~Gehring, G.~J.~Martin and B.~P.~Palka,}  {\em An introduction to the theory of higher-dimensional quasiconformal mappings}.
Mathematical Surveys and Monographs, 216. American Mathematical Society, Providence, RI, 2017.

\bibitem {GMSV05}
{\sc V.~Gutlyanski\u\i, O.~Martio, T.~Sugawa and M.~Vuorinen,}  {\em On the
degenerate Beltrami equation}, Trans. Amer. Math. Soc. \textbf{357}
(2005), no. 3, 875--900.

\bibitem {GSS13}
{\sc V.~Gutlyanski\u\i, K.~Sakan and T.~Sugawa,}  {\em On $\mu$-conformal homeomorphisms and boundary correspondence},
Complex Var. Elliptic Equ. \textbf{58} (2013), no. 7, 947--962.

\bibitem {HLM89}
{\sc D.~A.~Herron, X.~Y.~Liu and D.~Minda,}  {\em Ring domains with separating circles or separating annuli},
J. Analyse Math. \textbf{53} (1989), 233--252.

\bibitem {Ibr03}
{\sc Z.~Ibragimov,}  {\em M\"obius modulus of ring domains in $\overline{\mathbb R}^n$},
Ann. Acad. Sci. Fenn. Math. \textbf{28} (2003), no.~1, 193--206.

\bibitem{JV96}
{\sc P.~J\"arvi and M.~Vuorinen,}  \emph{Uniformly perfect sets and quasiregular
  mappings}, J. London Math. Soc. \textbf{54} (1996), 515--529.

\bibitem{SugawaUP}
{\sc T.~Sugawa,}  \emph{Various domain constants related to uniform perfectness},
  Complex Variables Theory Appl. \textbf{36} (1998), 311--345.

\bibitem{SugawaUPS}
 {\sc T.~Sugawa,}  \emph{Uniformly perfect sets: analytic and geometric aspects {{\rm
  (Japanese)}}}, Sugaku \textbf{53} (2001), 387--402, English translation in
  Sugaku Expo. {\bf 16} (2003), 225--242.

\bibitem{Sug10}
{\sc T.~Sugawa,}  {\em Modulus techniques in geometric function theory}, J. Anal. \textbf{18} (2010), 373--397.


\bibitem{Teich}
{\sc O.~Teichm\"uller, }  {\em Untersuchungen \"uber konforme und quasikonforme Abbildung},
Deutsche Math. {\bf 3} (1938), 621--678.
English translation ``Investigations on conformal and quasiconformal mappings" 
by M.~Brakalova-Trevithick and M.~Weiss is included in
{\em Handbook of Teichm\"uller Theory} (A.~Papadopoulos, ed.), Vol.~VII, pp.~463--530, EMS Publishing House,
Z\"urich, 2020.

\bibitem{Vai71}
{\sc J.~V\"ais\"al\"a,} {\em Lectures on n-dimensional quasiconformal
mappings}, Lecture Notes in Mathematics, Vol.~229. Springer-Verlag, Berlin-New York, 1971.


\bibitem {Zor17}
{\sc V.~A.~Zorich,} {\em Several remarks on multidimensional quasiconformal mappings}, Sb. Math. \textbf{208} (2017), no. 3-4, 377--398.

\end{thebibliography}
\end{document}